\documentclass[11pt,reqno]{amsart} 
\usepackage[height=190mm,width=130mm]{geometry}
\usepackage{amsmath}
\usepackage{amssymb,latexsym}
\usepackage{amsrefs}
\usepackage[draft=true]{hyperref}
\usepackage{cite}
\usepackage{a4wide}
\usepackage{amscd}
\usepackage{graphics}
\usepackage{color}
\usepackage[all]{xy}

\usepackage{txfonts}
\usepackage{amsthm}   
\usepackage{mathrsfs}

\theoremstyle{plain}
\newtheorem{theorem}{Theorem}

\newtheorem{corollary}{Corollary}
\newtheorem{proposition}{Proposition}

\theoremstyle{definition}

\theoremstyle{remark}
\newtheorem{remark}{Remark}
\newtheorem{example}{Example}
\newcommand{\Z}{\ensuremath{\mathbb{Z}}}   
\newcommand{\N}{\ensuremath{\mathbb{N}}}

\newcommand{\Hom}{\operatorname{Hom}}
\newcommand{\Tor}{\operatorname{Tor}}
\newcommand{\Ext}{\operatorname{Ext}} 

\newcommand{\Ker}{\operatorname{Ker}}

\newcommand{\Soc}{\operatorname{Soc}}

\newcommand{\J}{\operatorname{J}}
\newcommand{\E}{\operatorname{E}}

\numberwithin{equation}{section} 

\begin{document}
\title[Rings whose mininjective modules are injective]{Rings whose mininjective modules are injective}

\author[Y. Alag\"oz]{Yusuf Alag\"oz}
\address{Yusuf Alag\"oz \\ Department of Mathematics \\ Hatay Mustafa Kemal University\\  Hatay, Turkey}
\email{yusuf.alagoz@mku.edu.tr}

\author[S. Benl\.{i}-G\"{o}ral]{Sinem Benl\.{i}-G\"{o}ral}
\address{Sinem Benl\.{i}-G\"{o}ral \\ Department of Mathematics \\ \.{I}zmir Institute of Technology \\ \.{I}zmir, Turkey}
\email{sinembenli@iyte.edu.tr}

\author[E. B\"uy\"uka\c{s}{\i}k]{Engin B\"uy\"uka\c{s}{\i}k}
\address{Engin B\"uy\"uka\c{s}{\i}k \\ Department of Mathematics \\ \.{I}zmir Institute of Technology \\ \.{I}zmir, Turkey}
\email{enginbuyukasik@iyte.edu.tr}

\author[J. R. Garc\'{i}a Rozas]{Juan Ram\'{o}n Garc\'{i}a Rozas}
\address{Juan Ram\'{o}n Garc\'{i}a Rozas \\ Department of Mathematics \\ Almeria University\\ Almer\'{i}a, Spain}
\email{jrgrozas@ual.es}

\author[L. Oyonarte]{Luis Oyonarte}
\address{Luis Oyonarte \\ Department of Mathematics \\ Almeria University\\ Almer\'{i}a, Spain}
\email{oyonarte@ual.es}

\begin{abstract}
The main goal of this paper is to characterize rings over which the mininjective modules are injective, so that the classes of mininjective modules and injective modules coincide. We show that these rings are precisely those Noetherian rings for which every min-flat module is projective and we study this characterization in the cases when the ring is Kasch, commutative and when it is quasi-Frobenius. We also treat the case of $n\times n$ upper triangular matrix rings, proving that their mininjective modules are injective if and only if $n=2$.

We use the developed machinery to find a new type of examples of indigent modules (those whose subinjectivity domain contains only the injective modules), whose existence is known, so far, only in some rather restricted situations.
\end{abstract}



\maketitle

\section{Introduction}
Quasi-Frobenius rings ($QF$-rings, for short) were introduced by Nakayama \cite{nakayama, nakayama1} in the study of representations of algebras. Subsequently, $QF$-rings played a central role in ring theory and numerous characterizations were given by various authors (see for instance \cite{eilenberg-nakayama, faith, faith-walker, ikeda, ikeda-nakayama}). In particular, Ikeda \cite[Theorem 1]{ikeda} characterized these rings as the left (right) self-injective, left and right Artinian rings. Numerous investigations have been conducted to improve on Ikeda's previously mentioned result by weakening either the Artinian condition, or the injectivity condition, or both, and this has led to the discovery of new concepts for rings and modules such as mininjectivity and simple-injectivity, as well as numerous important studies on them (see for example \cite{bjork, harada, minflat, onminflat, maorings, universallymininjective}). In addition, some generalizations of the aforementioned rings and modules have also been studied in the literature (see  \cite{strongly-simple, soc-injective, almost-v, cl}).

The concept of mininjectivity for rings (in the Artinian case) first appeared as a property for characterizing $QF$-rings in the paper of Ikeda \cite{ikeda}, and it was shown that a ring $R$ is $QF$ if and only if it is right and left Artinian and right and left mininjective \cite[Theorem 2.30]{qfrings}. In 1982, Harada \cite{harada} introduced the notions of mininjective modules and rings as follows: if $M$ and $N$ are right $R$-modules, $M$ is said to be \textit{min-N-injective} if for every simple submodule $K$ of $N$ and every homomorphism $f: K \to M$, there exists a homomorphism $h: N \to M$ such that $h|_{K}=f$. If we let $N=R$ then $M$ is called \textit{mininjective}, that is, $\Ext^{1}_{R}(R/I, M)=0$ for every minimal right ideal $I$ of $R$.

Let us also remind the notion of min-flat modules introduced in \cite{minflat}. A left $R$-module $M$ is called \textit{min-flat} if $\Tor^R_1(R/I,M)=0$ for any minimal right ideal $I$ of $R$. With the help of the isomorphism $\Ext^{1}_{R}(R/I, M^{+})\cong \Tor_{1}^{R}(R/I, M)^{+}$ for any right ideal $I$ of $R$, it can be easily seen that a left $R$-module $M$ is min-flat if and only if $M^{+}$ is mininjective.

On the other hand, according to Harada \cite{harada-noteonalmostrelative}, $M$ is said to be \emph{simple-$N$-injective} if for every submodule $K$ of $N$, every homomorphism $f : K \rightarrow M$ with $f(K)$ simple extends to $N$. If $N=R$ then $M$ is called simple-injective.

Although mininjective modules and simple-injective modules have been extensively studied, the knowledge of these classes of modules is far from being near completion, with a large number of interesting problems remaining open that should be addressed. Recently, rings whose right simple-injective modules are injective have been characterized in \cite{onsimpleinj}, and this has been a motivation to devote this paper to the study of mininjective modules and to find conditions on the ring for its mininjective modules to be injective.

In the study of the rings whose mininjective right $R$-modules are injective, one specific question is particularly significant: when is a simple right $R$-module not isomorphic to any right ideal of $R$ injective?

Rings whose simple right $R$-modules not isomorphic to right ideals are all injective are called \emph{right quasi $V$-rings} and were introduced in \cite{onsimpleinj}. We devote Section 2 to develop a deep treatment of these rings.

The concept of right quasi $V$-rings generalizes those of right (generalized) $V$-ring and right Kasch ring. We give an example showing that being a quasi $V$-ring is not left-right symmetric and we prove that $R/J(R)$ is a right $GV$-ring precisely when $R$ is a right quasi $V$-ring.

\emph{Right almost $V$-rings} were introduced as the rings $R$ for which every simple right $R$-module is almost injective (see \cite{almost-v}), and it is a problem of great interest to know how right almost $V$-rings and quasi $V$-rings are related. This question is also addressed in Section 2, where we find conditions on the ring for the class of right quasi $V$-rings, the class of right Kasch rings and that of right almost $V$-rings coincide. Our interest in almost $V$-rings lies in \cite{cl}, where the authors characterize them by assuming that the rings are already commutative Noetherian quasi $V$-rings. But we can show that these results can be extended to more general commutative quasi $V$-rings (thus avoiding the noetherianity condition).

In Section 3 we address the first natural question that comes to mind right after introducing the concept of mininjectivity: when do the classes of mininjective and injective modules coincide? Following \cite{strongly-min-coherent}, given a class $\mathfrak{C}$ of finitely presented right $R$-modules, the ring $R$ is called \textit{right strongly $\mathfrak{C}$-coherent} if whenever $0\rightarrow K\rightarrow P\rightarrow C\rightarrow 0$ is exact, $C\in \mathfrak{C}$ and $P$ is finitely generated projective, $K$ is $\mathfrak{C}$-projective. We define \textit{right strongly min-coherent} rings by taking $\mathfrak{C}=\{R/I:\ I\ \mbox{is a minimal right ideal of}\ R\}$. These rings play a crucial role in studying when mininjectivity implies injectivity. In fact, in Theorem 3.2 we give a complete general characterization of rings whose mininjective modules are all injective, and it can be seen that these rings must necessarily be, among other things, strongly min-coherent rings. In this Section 3 we also give other deeper characterizations of these rings (rings over which mininjectivity implies injectivity) in some particular cases. We show that if every mininjective right $R$-module is injective then $R$ is a right Artinian, right strongly min-coherent and a right quasi $V$-ring. Over a right $PS$-ring we obtain that if every mininjective right $R$-module is injective then $R$ is a right hereditary, right Artinian and a right generalized $V$-ring. We also prove that if $R$ is a right Kasch ring then every mininjective right $R$-module is injective if and only if $R$ is right Artinian and right strongly min-coherent. In particular, a commutative ring $R$ is strongly min-coherent and Artinian if and only if every mininjective $R$-module is injective. We also treat the case of $QF$-rings: in the proof of Ikeda's Theorem given in \cite[Theorem 2.30]{qfrings}, it is essential that $R$ be two-sided mininjective (see \cite[Bj\"ork Example 2.5]{qfrings}). However, in this paper we give a characterization of $QF$-rings using only the right mininjectivity: $R$ is $QF$ if and only if every mininjective right R-module is projective if and only if $R$ is right mininjective, right Artinian, right strongly min-coherent and right Kasch.

Universally mininjective rings are also not forgotten. We prove that $R$ is right universally mininjective if and only if any simple right module is mininjective.

We conclude Section 3 by treating the case of upper triangular matrix rings. We obtain that over the ring $R=UT_{n}(k)$, where $k$ is a field, every min-injective right $R$-module is injective if and only if $n=2$.

We now set the general assumptions, terminology and notation that will be used throughout the paper. The letter $R$ will always stand for an associative ring with identity $1 \neq 0$, and modules are unital $R$-modules. For a module $M$, its character module, $\Hom_{\mathbb{Z}}(M,\mathbb{Q/Z})$, is denoted by $M^{+}$. The notations $\J(M)$, $\Soc(M)$ and $\E(M)$ are used for the Jacobson radical, the socle and the injective envelope of $M$, respectively.

\section{Quasi $V$-rings}

We start by recalling the well known concepts of $V$-rings and Kasch rings. While a ring is said to be a \emph{right $V$-ring} if all its simple right modules are injective, the ring is \emph{right Kasch} if every simple right module embeds in the ring itself.

As a generalization of right $V$-rings, \emph{right quasi $V$-rings} were defined in \cite{onsimpleinj} as those rings whose simple right $R$-modules which are not isomorphic to a right ideals are all injective. Left quasi $V$-rings are defined similarly, and a left quasi $V$-ring needs not be a right quasi $V$-ring (see Example \ref{leftbutright}).

\begin{remark} \label{quasiVbutnotkasch}
Every right Kasch ring is a right quasi $V$-ring, but the converse is not true in general. Let $R=\prod_\Gamma F$ be an infinite direct product of copies of a field. Then, $R$ is a non-semisimple commutative regular ring and therefore it is a (quasi) $V$-ring, but $R$ is not Kasch since otherwise it would be semisimple.
\end{remark}

In the next result we provide a situation in which the notions of Kasch ring and quasi $V$-ring coincide. A ring $R$ is said to be a \emph{right $N$-ring} if each maximal right ideal of $R$ is finitely generated. Right Noetherian rings are trivial examples of right $N$-rings, but not all right $N$-rings have to be Noetherian as can be seen with ring $R=C^{\infty}[0,1]$ of all smooth (infinitely differentiable) functions on $[0,1]$. Every maximal ideal $M$ of $R$ is of the form $\{ f \in R: f(c)=0\}$ where $c$ is a uniquely determined point in $[0,1]$, and all maximal ideals are principal (see \cite[Corollary 28]{nonnoether}). However, $R$ is not Noetherian. To see this, consider the chain $I_1 \subseteq I_2 \subseteq \dots \subseteq I_n \subseteq I_{n+1} \subseteq \cdots $ of ideals of $R$ where $I_n = \{ f \in R : f=0 \hspace{0.2 cm} \text{on} \hspace{0.2 cm} [0, \frac{1}{n}]\}$. Since the bump function
\begin{equation*}
b_{n+1}(x)=
    \begin{cases}
        0 & \text{if } x \in [0, \frac{1}{n+1}]\\
        e^{-\frac{1}{(x-\frac{1}{n+1})^2}} & \text{if } x \in (\frac{1}{n+1}, 1]
    \end{cases}
\end{equation*}
is contained in $I_{n+1}$ but not in $I_n$, we have that $I_n \neq I_{n+1}$ for every $n \in \N$.

\begin{proposition} \label{commkasch}
Let $R$ be a commutative $N$-ring. Then, $R$ is a quasi $V$-ring if and only if $R$ is a Kasch ring.
\end{proposition}
\begin{proof}
We only need to prove the necessary condition, so assume that there exists a simple $R$-module $S$ which is not isomorphic to any ideal of $R$. Then, $S$ is injective by the hypothesis and hence $S$ is flat by \cite[Lemma 2.6]{ware}.

On the other hand, $S\cong R/I$ for some is a maximal right ideal $I$ of $R$, and since $I$ is finitely generated, $R/I$ is finitely presented. But then $S$ is both flat and finitely presented, that is, $R/I$ is projective (\cite[Theorem 4.30]{lam}), so the short exact sequence $0\rightarrow I\rightarrow R\rightarrow R/I \rightarrow 0$ splits. This gives that $S$ is isomorphic to an ideal of $R$, which is a contradiction. Therefore, every simple $R$-module is isomorphic to a minimal ideal of $R$, and so $R$ is a Kasch ring.
\end{proof}

Recall that a ring $R$ is said to be a \emph{right generalized $V$-ring} (shortly, \emph{right $GV$-ring}) if every simple right $R$-module is either injective or projective. Thus, since every projective simple right $R$-module embeds in $R$, over a right $GV$-ring every simple right $R$-module that is not isomorphic to a right ideal of $R$ must be injective. Hence, every right $GV$-ring is a right quasi $V$-ring and we have the following.

\begin{corollary} \label{gv}
Let $R$ be a ring. The following are equivalent.
\begin{enumerate}
\item $R$ is a right $GV$-ring.
\item $R$ is a right quasi $V$-ring and $\Soc(R_{R})$ is projective.
\end{enumerate}
\end{corollary}

\begin{example} \label{quasiVbutnotGV}
For an example of a right quasi $V$-ring which is not a right $GV$-ring, consider the commutative Artinian ring $R=\Z/p^{3}\Z$ where $p$ is a prime number. As $R$ is Kasch, $R$ is a quasi $V$-ring. Note that $\Soc(R)$ is neither injective nor projective, so $R$ is not a $GV$-ring by Corollary \ref{gv}.
\end{example}

The following example shows that a left quasi $V$-ring needs not be a right quasi $V$-ring.

\begin{example} \label{leftbutright}
We know there are examples of left $GV$-rings with projective right socle which are not right $GV$-rings (see for example \cite[Section 4, Example (d)]{baccella}), so let $R$ be one such ring. Then, $R$ turns out to be a left quasi $V$-ring which is not a right quasi $V$-ring by Corollary \ref{gv}. Thus, we see there exist left quasi $V$-rings which are not right quasi $V$-rings.
\end{example}

We now show with an example that the class of quasi $V$-rings is not closed under quotients.

\begin{example}
Let $R$ be a commutative Noetherian local ring with $\Soc(R)=0$. Call $M$ its maximal ideal and consider $S=R/M$ and the trivial extension $\tilde{R}$ of $R$, that is, $$\tilde{R}= R \ltimes S = \Biggl\{ \left( \begin{array}{cc} r & x \\ 0 & r \end{array} \right)  : r \in R, x \in S \Biggl\}$$ with the ordinary matrix operations. Then, $\tilde{R}$ is a commutative local ring whose maximal ideal is $\tilde{M}= M \ltimes S$. Moreover, since $\tilde{R}/\tilde{M} \cong R/M =S \cong 0 \ltimes S$, we have that $\tilde{R}$ is a Kasch ring and hence a quasi $V$-ring.

Now, the assumption $\Soc(R)=0$ implies that the intersection of all essential ideals of $R$ is zero, but if $I$ is any essential ideal of $R$, then the ideal $$\tilde{I}= \Biggl\{ \left( \begin{array}{cc} a & x \\ 0 & a \end{array} \right): a \in I, x \in S \Biggl\}$$ is an essential ideal of $\tilde{R}$. Thus, $$\Soc(\tilde{R})= \bigcap \{ \tilde{I}: I \text{\hspace{0.09 cm}is essential in}\hspace{0.09 cm} R\} = \left( \begin{array}{cc} 0 & S \\ 0 & 0 \end{array} \right)= 0 \ltimes S.$$

Since $\tilde{R}/\Soc(\tilde{R})\cong R$ and $R$ has zero socle, $\tilde{R}/\Soc(\tilde{R})$ is not Kasch. Thus, $\tilde{R}/\Soc(\tilde{R})$ is not a quasi $V$-ring by Proposition \ref{commkasch}.
\end{example}

\begin{proposition}
Let $R$ be a right quasi $V$-ring. Then, $R/ \J(R)$ is a right $GV$-ring.
\end{proposition}
\begin{proof} Let $S$ be a simple right $R/\J(R)$-module. If $S$ can be embedded in $R/\J(R)$ then $S\cong U$ for some minimal right ideal $U$ of $R/\J(R)$. Since $\J(R/\J(R))=0$, $U$ is not small in $R/\J(R)$. Thus, there is a maximal right ideal, say $K$, of $R/\J(R)$ such that $U+K=R/\J(R)$. As $U$ is minimal, $U \cap K=0$, hence $R/\J(R)=U\oplus K$. In particular, $U$ and so $S$ are projective right $R/\J(R)$-modules.

On the other hand, suppose that $S$ cannot be embedded in $R/\J(R)$. Note that $S$ is also a simple right $R$-module, so it cannot be embedded in $R$ either. By the quasi $V$-ring assumption on $R$ we have that $S$ is an injective right $R$-module, and so $S$ is injective as a right $R/\J(R)$-module. Therefore, $R/\J(R)$ is a right $GV$-ring.
\end{proof}

In \cite{almost-v}, a ring $R$ is said to be a \emph{right almost $V$-ring} if every simple right $R$-module is almost injective, and a right $R$-module $M$ is said to be \emph{almost injective} (\cite{baba}) if for every embedding $i: A \to B$ of right $R$-modules and every homomorphism $f: A \to M$, either there exists a homomorphism $g: B \to M$ such that $gi=f$, or there exists a non-zero direct summand $D$ of $B$ and a homomorphism $h: M \to D$ such that $hf=\pi i$, where $\pi:B\to D$ stands for the canonical projection. Right $V$-rings are trivial examples of right almost $V$-rings.

\begin{remark}\label{quasiVbutnotalmostV}
Every right almost $V$-ring is a right quasi $V$-ring by \cite[Proposition 2.2]{almost-v}. However, the converse is not true in general. Let $R=\Z /p^n \Z$ for some prime integer $p$ and natural number $n$. Then, $R$ is a $QF$-ring and so is a Kasch ring, whence $R$ is a quasi $V$-ring. But if $n>2$ $R$ is not an almost $V$-ring by \cite[Theorem 3.1]{almost-v}.
\end{remark}

The following corollary yields a situation where quasi $V$-rings are almost $V$-rings.

\begin{corollary} \label{almostv}
The following are equivalent for a commutative $N$-ring $R$.
\begin{enumerate}
\item $R$ is an almost $V$-ring.
\item $R$ is a quasi $V$-ring and $\Soc(R)$ is almost injective.
\end{enumerate}
\end{corollary}
\begin{proof} $(1) \Rightarrow (2)$ Almost $V$-rings are quasi $V$-rings. Then, $R$ is a Kasch ring by Proposition \ref{commkasch}. Thus, $R$ is Noetherian by \cite[Theorem 3.6]{cl} and so $\Soc(R)$ is almost injective by \cite[Corollary 3.8]{cl}. This proves $(1)$.

$(2) \Rightarrow (1)$ Assuming $(2)$, we have that $R$ is a Kasch ring by Proposition \ref{commkasch}. Then, $R$ is an almost $V$-ring by \cite[Corollary 3.8]{cl}.
\end{proof}

In \cite[Proposition 3.4]{cl}, the authors gave a characterization of commutative Noetherian almost $V$-rings. In the following proposition we extend their result to commutative almost $V$-rings whose maximal ideals are finitely generated.

\begin{proposition} \label{cl}
The following statements are equivalent for a commutative $N$-ring $R$.
\begin{enumerate}
\item $R$ is an almost $V$-ring.
\item $R$ is a quasi $V$-ring and $\Soc(R)$ is almost injective.
\item $R$ is a quasi-Frobenius serial ring with $\J(R)^{2}=0$.
\item $R=\prod_{i=1}^{n}R_{i}$, where $R_{i}$ is either a field or a quasi-Frobenius ring of length $2$.
\end{enumerate}
\end{proposition}
\begin{proof}
$(1)\Leftrightarrow(2)$ follows from Corollary \ref{almostv}.

$(3)\Leftrightarrow (4)\Rightarrow (1)$ are clear by \cite[Theorem 3.4]{cl}

$(1)\Rightarrow (3)$ Let $R$ be an almost $V$-ring. Since $R$ is a quasi $V$-ring, $R$ is a Kasch ring by Proposition \ref{commkasch}, so $R$ is Noetherian by \cite[Theorem 3.6]{cl}. Thus, (3) follows by \cite[Proposition 3.4]{cl}.
\end{proof}

\section{Rings whose mininjective right modules are injective}

Clearly, injective modules are always mininjective. However, the converse of this fact is not true in general since, for example, every abelian group is mininjective whereas $\Z$ is not injective as an abelian group. In this section we give some conditions which guarantee that each mininjective right $R$-module is injective.

Let $\mathfrak{C}$ be a class of finitely presented right $R$-modules. Recall that a right $R$-module $M$ is called \textit{$\mathfrak{C}$-injective} if $\Ext^1_R(C,M)=0$ for every $C \in \mathfrak{C}$ (\cite{zhu2013}). Dually, $M$ is called \textit{$\mathfrak{C}$-projective} if $\Ext^1_R(M,N)=0$ for any $\mathfrak{C}$-injective right $R$-module $N$ (\cite{strongly-min-coherent}). Following \cite{strongly-min-coherent}, a ring is called \textit{right strongly $\mathfrak{C}$-coherent} if, whenever $0\rightarrow K\rightarrow P\rightarrow C\rightarrow 0$ is exact with $C\in \mathfrak{C}$ and $P$ finitely generated projective, the module $K$ is $\mathfrak{C}$-projective. The ring is called right min-coherent if every minimal right ideal is finitely presented.

We now set some notation. From now on, the symbol $\mathscr{C}$ will denote the set of right $R$-modules $$\mathscr{C}=\{R/I: I\ \mbox{is a minimal right ideal of}\ R\},$$ and this set will be use to define right strongly min-coherent rings. Indeed, the ring $R$ will be said to be \textit{right strongly min-coherent} if it is right strongly $\mathscr{C}$-coherent. Of course, a right $R$-module $M$ is $\mathscr{C}$-injective if and only if it is mininjective, and $M$ is called \textit{min-projective} if $\Ext^{1}_{R}(M,N)=0$ for every mininjective right $R$-module $N$.

From \cite[Theorem 1]{strongly-min-coherent} and its proof we have the following proposition.

\begin{proposition} \label{stronglymincoherent}
Let $R$ be a ring. Then, the following statements are equivalent.
\begin{enumerate}
\item $R$ is right strongly min-coherent.
\item Every minimal right ideal of $R$ is min-projective.
\item $\Ext^{2}_R(R/I,M)=0$ for every mininjective right $R$-module $M$ and every minimal right ideal $I$ of $R$.
\item For every short exact sequence $0\rightarrow M\rightarrow N\rightarrow L\rightarrow 0$ of right $R$-modules with $M$ and $N$ mininjective, $L$ is mininjective.
\item For each mininjective right $R$-module $M$, $\E(M)/M$ is mininjective.
\end{enumerate}
\end{proposition}

In the following theorem we characterize the rings (and their categories of modules) over which all mininjective right modules are injective.

\begin{theorem}  \label{mininjectivesareinjchar}
Let $R$ be a ring. The following statements are equivalent.
\begin{enumerate}
\item Every mininjective right $R$-module is injective.
\item Every right $R$-module is min-projective.
\item $R$ is right Noetherian and every min-flat left $R$-module is projective.
\item \begin{enumerate}\item[(i)] $R$ is a right Artinian right strongly min-coherent right quasi $V$-ring.
\item[(ii)] For any simple right module $S$, either $S\subseteq R$ or $S\subseteq R/I$ for some minimal right ideal $I$ of $R$.
\end{enumerate}
\end{enumerate}
\end{theorem}
\begin{proof}
$(1)\Leftrightarrow (2)$ is obvious from the definition of
min-projectivity.

$(1)\Rightarrow (4)$ (i) Let $\{E_{\gamma}\}_{\gamma \in \Gamma}$ be an arbitrary family of injective right $R$-modules. Then, the module $\bigoplus_{\gamma \in \Gamma}E_{\gamma}$ is mininjective and so injective by the hypothesis, so $R$ is right Noetherian.

Now, let $A$ be any cyclic right $R$-module with $\Soc(A) = 0$. Then, $\Hom(I,B)=0$ for each minimal right ideal $I$ of $R$ and each submodule $B$ of $A$, which means that every submodule of $A$ is mininjective and so injective by assumption. Therefore, every submodule of $A$ is a direct summand, that is, $A$ is semisimple, and then $A=0$ necessarily.

We have just shown that every nonzero right $R$-module must contain a simple submodule, so $R$ is right semiartinian. Thus, $R$ must be right Artinian.

A simple application of Proposition \ref{stronglymincoherent} gives that $R$ is right strongly min-coherent since $\Ext^{1}_{R}(I,M)=0$ for every mininjective (so injective) right $R$-module and every minimal right ideal $I$ of $R$.

It only remains to be shown that $R$ is a right quasi $V$-ring. For let $S$ be a simple right $R$-module which is not isomorphic to any right ideal of $R$ and choose any minimal right ideal $I$ of $R$. We have that $\Hom(I,S)=0$ so $S$ is mininjective, and this implies that $S$ is injective by the hypothesis. Thus, $R$ is a right quasi $V$-ring.

(ii) Let $S$ be a simple right $R$-module. By the assumption $S$ is min-projective, i.e. ${S\in ^{\perp}(\mathscr{C}^{\perp})}$. This means that $S$ is given by a filtration of modules in $\mathscr{C}\cup \{ R\}$ (see \cite[Corollary 3.2.4]{gobel}) in the following way: there exists an ordinal number $\mu$ and a continuous ascending chain of modules $\mathcal{A}=(A_{\alpha} | \alpha\leq\mu)$ such that $A_{\alpha+1}/A_{\alpha}$ is isomorphic to an element of $\mathscr{C}\cup \{ R\}$ and that $S\oplus T=\cup_{\alpha\leq \mu} A_{\alpha}$ for some module $T$. Let $\alpha$ be the first ordinal number such that $S\subseteq
A_{\alpha}$ (so $\alpha$ is necessarily a successor ordinal number). Then,
$S\nsubseteq A_{\alpha-1}$ and $S\subseteq A_{\alpha}$, so, up to an isomorphism, we have $S\subseteq A_{\alpha}/A_{\alpha-1}\in \mathscr{C} \cup \{ R\}$.

$(1)\Rightarrow(3)$ As in the proof of $(1)\Rightarrow (4)$ we obtain that $R$ is right Artinian. Thus, $R$ is in particular right Noetherian and left perfect. But, Noetherian implies min-coherent, so we can apply \cite[Proposition 4.8]{minflat} to get that every min-flat left $R$-module is flat. Since $R$ is left perfect, flat left $R$-modules are projective and we are done.

$(3)\Rightarrow(1)$ By \cite[Proposition 4.8]{minflat} every mininjective right $R$-module $M$ is $FP$-injective, that is, $\Ext^1_R(F,M)=0$ for every finitely presented right $R$-module $F$. By the noetherianity of $R$, we have that $M$ is injective.

$(4)\Rightarrow (1)$ Let $M$ be a mininjective right $R$-module and $S$ be a simple right $R$-module. Since $R$ is right Artinian, to show that $M$ is injective it is enough to show that $\Ext^{1}(S,M)=0$.

Now, if $S$ is isomorphic to a simple left ideal of $R$ then $\Ext^{1}(S,M)=0$ by the strongly min-coherence of $R$ (see Proposition \ref{stronglymincoherent}).

If, on the contrary, $S$ is not isomorphic to a simple left ideal of $R$, then $S$ is injective by the hypothesis. But also by the hypothesis there exists a minimal right ideal $I$ of $R$ such that $S\subseteq R/I$. Therefore, the injectivity of $S$ gives that $S$ is, up to an isomorphism, a direct summand of the minprojective right module $R/I$. Thus, $S$ is itself minprojective and hence $\Ext^{1}(S,M)=0$.
\end{proof}

In a recent paper, Aydogdu and L\'opez-Permouth defined $N$-subinjective modules as those $M$ for which every homomorphism $N \rightarrow M$ extends to some $E(N) \rightarrow M$. For a given module $M$, its subinjectivity domain, $\underline{In}^{-1}(M)$, is defined as the collection of all modules $N$ such that $M$ is $N$-subinjective. If $N$ is injective, then $M$ is vacuously $N$-subinjective. So, the smallest possible subinjectivity domain is the class of all injective modules. A module with such a subinjectivity domain was defined in \cite{indigent} as indigent and the existence of indigent modules for an arbitrary ring is unknown. There are examples of rings over which indigent modules do exist, for example $\mathbb{Z}$ and Artinian serial rings (see \cite{indigent}), but other than that, little is known about these type of rings. With the new tools we have developed in hand, we can now provide a new class of rings over which indigent modules are guaranteed to exist.

\begin{theorem}
Let $R$ be a ring and $\Lambda=\{R/aR: Ra\ \mbox{is a minimal left ideal of}\ R\}$. The following statements are equivalent.
\begin{enumerate}
\item Every min-injective left $R$-module is injective.
\item $F=\prod_{A\in\Lambda}A^{+}$ is indigent and $R$ is left min-coherent.
\end{enumerate}
\end{theorem}
\begin{proof}
$(1)\Rightarrow(2)$ Let $N \in \underline{In}^{-1}(F)$. Then, $N \in \underline{In}^{-1}(A^{+})$ for every $A \in \Lambda$, that is, $A^{+}$ is $N$-subinjective for every $A \in \Lambda$.

Now, consider the exact sequence  $0 \rightarrow N \rightarrow E(N) \rightarrow  E(N)/N \rightarrow  0$.

Since $A^{+}$ is $N$-subinjective, the rows of the following commutative diagram, where the vertical isomorphisms are obtained by applying twice the adjunction $(\otimes, \Hom)$, are exact: $$\xymatrix@1{0 \ar[r] & \Hom(E(N)/N,A^{+}) \ar[r] \ar[d]_{\cong} & \Hom(E(N),A^{+}) \ar[r] \ar[d]_{\cong} & \Hom(N,A^{+}) \ar[r] \ar[d]_{\cong} & 0 \\ 0 \ar[r] & (A \otimes E(N)/N)^{+} \ar[r] \ar[d]_{\cong} & (A\otimes E(N))^{+} \ar[r] \ar[d]_{\cong} & (A\otimes N)^{+} \ar[r] \ar[d]_{\cong} & 0 \\ 0 \ar[r]  & \Hom(A,(E(N)/N)^{+}) \ar[r] & \Hom(A,E(N)^{+}) \ar[r] &  \Hom(A,N^{+}) \ar[r] & 0}$$

Being the third row exact means that the sequence $$0 \rightarrow (E(N)/N)^{+} \otimes R/S \rightarrow E(N)^{+} \otimes R/S \rightarrow N^{+} \otimes R/S \rightarrow 0$$ is exact for any minimal left ideal $S$ by \cite[Lemma 2.1]{onminflat}. Hence ($R/S$ is finitely presented) we get the commutative diagram with exact rows $$\xymatrix@1{0 \ar [r]  & (E(N)/N)^{+} \otimes R/S  \ar [r] \ar [d]_{\cong} & E(N)^{+} \otimes R/S  \ar [r] \ar [d]_{\cong} &  N^{+} \otimes R/S  \ar [r] \ar [d]_{\cong}& 0 \\ 0  \ar [r] & \Hom (R/S, E(N)/N)^{+}    \ar [r] &    \Hom (R/S, E(N))^{+}  \ar [r]    & \Hom (R/S, N)^{+} \ar [r] & 0}$$ from which it follows that $$0 \to \Hom(R/S, N) \to \Hom(R/S, E(N)) \to \Hom(R/S, E(N)/N) \to 0$$ is an exact sequence for any minimal left ideal $S$, whence $N$ is mininjective. Thus, by the hypothesis, $N$ is injective.

Then, the left min-coherence of $R$ follows by Theorem \ref{mininjectivesareinjchar} since $R$ must be left Noetherian.

$(2)\Rightarrow(1)$ Let $N$ be a mininjective left $R$-module and consider the exact sequence $$0 \to N \to \E(N) \to \E(N)/N \to 0.$$ We claim that $$0 \to \Hom (\E(N)/N, F) \to \Hom ( \E(N),F) \to \Hom(N, F) \to 0$$ is also exact.

Since $N$ is mininjective and $R$ is left min-coherent, $N^{+}$ is min-flat by \cite[Theorem 4.5]{minflat}. Thus, $0 \to (\E(N)/N)^{+} \to \E(N)^{+} \to N^{+} \to 0$ is min-pure by \cite[Proposition 2.2]{onminflat} and then $$0 \to \Hom(A, (\E(N)/N)^+) \to \Hom(A, \E(N)^{+}) \to \Hom(A, N^{+}) \to 0$$ is exact for any $A \in \Lambda$. But every $A\in \Lambda$ is finitely presented, so we get the commutative diagram with exact rows $$\xymatrix@1{0 \ar [r]  & (\Hom(A,N^{+}))^{+}  \ar [r] \ar [d]_{\cong} & (\Hom(A,\E (N)^{+}))^{+}  \ar [r] \ar [d]_{\cong} &  (\Hom (A,(\E (N)/N)^{+})^{+}   \ar [r] \ar [d]_{\cong}& 0 \\ 0 \ar[r] & A \otimes N^{++}    \ar [r] &   A \otimes \E(N) ^{++}  \ar [r] & A \otimes (\E(N)/N)^{++}  \ar [r] & 0}$$ which induces the commutative diagram with exact rows $$\xymatrix@1{0 \ar [r]  & (A \otimes (\E (N)/N)^{++} )^{+}   \ar [r] \ar [d]_{\cong} & (A \otimes \E (N)^{++} )^{+}   \ar [r] \ar [d]_{\cong} &  (A \otimes N^{++} )^{+}   \ar [r] \ar [d]_{\cong}& 0 \\ 0 \ar[r] & \Hom( (\E(N)/N)^{++}, A^{+})    \ar [r] \ar[d]_{h} &    \Hom( \E(N)^{++}, A^{+})  \ar [r]^{g}  \ar[d]_{\psi} & \Hom( N^{++}, A^{+})  \ar [r]  \ar[d]_{f} & 0 \\ 0 \ar [r] & \Hom (\E(N)/N, A^{+}) \ar [r]\ar[d]  & \Hom(\E(N), A^{+}) \ar [r]^{\phi}\ar[d]  & \Hom(N, A^{+}) \ar [r] \ar[d] & 0 \\ & 0 & 0 & 0 &}$$

Since every module is pure in its double dual and $A^{+}$ is pure-injective, we see that $h$, $\psi$ and $f$ are epimorphisms. But $g$ is also epic so $fg=\phi \psi$ is epic and then $\phi$ is necessarily an epimorphism (for any $A \in \Lambda$). Hence, $$0 \to \Hom(\E(N)/N, F) \to \Hom(\E(N), F) \to \Hom(N, F) \to 0$$ is exact, that is, $F$ is $N$-subinjective, and then $N$ is injective since $F$ is indigent.
\end{proof}

When studying the injectivity of mininjective modules we have observed that heritability plays a role. Recall that a ring $R$ is said to be right hereditary if every right ideal of $R$ is projective. The ring is called \emph{right $PS$} if any of its minimal right ideals is projective, or equivalently, if its right socle is projective.

\begin{proposition}\label{ps}
Let $R$ be a right PS ring. If every mininjective right $R$-module is injective then $R$ is a right hereditary, right Artinian, right GV ring.
\end{proposition}
\begin{proof}
By Theorem \ref{mininjectivesareinjchar} $R$ is right Artinian and by \cite[Theorem 5.8]{minflat} every quotient of an injective right $R$-module is injective since the classes of mininjective modules and injective modules coincide by the hypothesis. Thus, \cite[Theorem 3.22]{lam} guarantees that $R$ is right hereditary.

Let now $S$ be any simple right $R$-module. If $S$ is isomorphic to a right ideal of $R$ then $S$ is projective by the $PS$ assumption, and if $S$ is not, then for any minimal right ideal $I$ of $R$ we have $\Hom(I,S)=0$. Thus, $S$ is mininjective and so injective by the hypothesis, and therefore $R$ is a right $GV$-ring.
\end{proof}

\begin{proposition} \label{qf-serial}
Let $R$ be an Artinian serial ring with $J^{2}(R)=0$. Then, every mininjective right $R$-module is injective.
\end{proposition}
\begin{proof}
Since $R$ is serial there exists a complete set of pairwise orthogonal idempotents, $I=\{e_{1},...,e_{n}\}$, such that each $e_{i}R$ is serial. Moreover, since $R$ is right Artinian, the complete set of simple right $R$-modules is given by the set $$\left\{ \frac{e_{1}R }{ e_{1}J(R)}, \frac{e_{2}R} { e_{2}J(R)}, \dots, \frac{e_{n}R}{ e_{n}J(R)}\right\}.$$ Then, any simple right $R$-module is isomorphic to some ${\displaystyle\frac{e_{i}R} { e_{i}J(R)}}$.

Let now $M$ a mininjective right $R$-module and $f:e_{i}J(R)\rightarrow M$ be any homomorphism. $M$ is of course min-$e_{i}R$-injective so there exists a homomorphism $g:e_{i}R\rightarrow M$ such that $gi=f$, where $i:e_{i}J(R) \rightarrow e_{i}R$ is the inclusion map. Therefore, $\Ext^{1}_{R}\left({\displaystyle\frac{e_{i}R} { e_{i}J(R)}},M\right)=0$, whence $M$ is max-injective. Thus, $M$ is injective by the perfectness of $R$.
\end{proof}

\begin{corollary} \label{mininjectivesareinj2}
Let $R$ be a ring whose minimal right ideals are almost-injective. Then, every mininjective right $R$-module is injective if and only if $R$ is a right Artinian serial ring with $J^{2}(R)=0$.
\end{corollary}
\begin{proof}
We only need to prove the necessary condition. For we apply Theorem \ref{mininjectivesareinjchar} to get that $R$ is right Artinian and also a right quasi V-ring. Then, every simple right module which is not isomorphic to an ideal is injective, and by the hypotheses every minimal right ideal is almost-injective, so we see that $R$ is indeed a right almost $V$-ring. Thus, the result follows by \cite[Corollary 3.5]{almost-v}.
\end{proof}

In Theorem \ref{mininjectivesareinjchar} we gave a characterization of rings whose mininjective modules are injective. It turns out that with the help of $C$-rings, such characterization can be reduce to a much simpler one, when the ring is assumed to be Kasch. Recall that a ring $R$ is said to be a \textit{right $C$-ring} if the right $R$-module $R/I$ has non-zero socle for every proper essential right ideal $I$ of $R$ (\cite{neat}). Left perfect rings and right semi-Artinian rings are well known examples of right $C$-rings. Right $C$-rings are exactly the rings whose max-injective right $R$-modules are injective, that is, the rings over which every right $R$-module $M$ satisfying $\Ext_R^1(S,M)=0$ for every simple right $R$-module $S$ is injective (\cite[Lemma 4]{smith}).

\begin{corollary} \label{mininjectivesareinj}
Let $R$ be a right Kasch ring. The following statements are equivalent.
\begin{enumerate}
\item Every mininjective right $R$-module is injective.
\item $R$ is right Artinian and right strongly min-coherent.
\item $R$ is a right $C$-ring right strongly min-coherent.
\end{enumerate}
\end{corollary}
\begin{proof}
$(1)\Rightarrow (2)$ follows by Theorem \ref{mininjectivesareinjchar}.

$(2)\Rightarrow(3)$ is clear since right Artinian rings are right $C$-rings.

$(3)\Rightarrow (1)$ Let $M$ be a mininjective right $R$-module. Since $R$ is both right Kasch and right strongly min-coherent we have that $\Ext^{1}_{R}(S,M)= 0$ for any simple right $R$-module $S$, that is, $M$ is max-injective, so $M$ is injective since $R$ a right $C$-ring.
\end{proof}

Mininjective rings play an important role when studying the structure of quasi-Frobenius rings. Indeed, Ikeda’s Theorem states that $R$ is quasi-Frobenius if and only if it is right Artinian and right and left mininjective (see \cite[Theorem 2.30]{qfrings}). It is important to highlight that in the proof of this characterization, the two-sided character of the mininjectivity of $R$ is essential as can be seen in \cite[Bj\"ork Example 2.5]{qfrings}. In the following result we provide an Ikeda-type characterization of QF-rings avoiding this two-sided constraint so far existing.

\begin{theorem} \label{qfring}
Let $R$ be a ring. Then the following statements are equivalent.
\begin{enumerate}
\item $R$ is a quasi-Frobenius ring.
\item Every mininjective right $R$-module is projective.
\item $R_{R}$ is mininjective and every mininjective right $R$-module is injective.
\item $R$ is right mininjective, right Artinian, right strongly min-coherent and right Kasch.
\item Every min-flat left $R$-module is injective.
\end{enumerate}
\end{theorem}
\begin{proof}
$(1)\Rightarrow(2)$ and $(1)\Rightarrow(3)$ Assume that $R$ is a Quasi Frobenius ring. It is obvious that $R$ is a right mininjective ring. Let $M$ be a mininjective right $R$-module and $f:P\rightarrow M$ an epimorphism with $P$ projective. We claim that $\Ker(f)$ is closed in $P$. Let $S$ be a simple right $R$-module and $h:S\rightarrow M$ a homomorphism. Since $R$ is right Kasch, $S$ can be embedded in $R$ via $i: S \rightarrow R$. Also, since $M$ is mininjective, there exists a homomorphism $g:R\rightarrow M$ such that $gi=h$. Then, projectivity of $R$ implies that there exists a homomorphism $\phi:R\rightarrow P$ such that $f\phi=g$. Thus, $h=gi=f\phi i$, where $\phi i:S\rightarrow P$, and this implies that $\Ker(f)$ is neat in $P$ by \cite[Definition 1.1]{neat}. Since $R$ is right Artinian, $R$ is a right C-ring, and so $\Ker(f)$ is closed in $P$ by \cite[Theorem 1.1]{neat}. On the other hand, being $R$ quasi-Frobenius implies that $P$ is injective. Closedness of $\Ker(f)$ in $P$ implies that the epimorphism $f:P\rightarrow M$ splits. Thus, $M$ is both projective and injective.

$(2)\Rightarrow(1)$ Every injective module is mininjective.

$(3)\Rightarrow(1)$ $R$ is right Artinian by Theorem \ref{mininjectivesareinjchar}.

$(3)\Leftrightarrow (4)$ Apply Theorem \ref{mininjectivesareinjchar} and Corollary \ref{mininjectivesareinj}.

$(1)\Rightarrow(5)$ For any min-flat left $R$-module $M$, $M^{+}$ is a mininjective right $R$-module, so the proof of $(1)\Rightarrow(2)$ shows that $M^{+}$ is projective. Since $R$ is left Noetherian, $M$ is injective by \cite[Theorem 2]{chetnam}.

$(5)\Rightarrow(1)$ Every projective left $R$-module is min-flat.
\end{proof}

It is well known that commutative Artinian rings are Kasch. Thus, the following result is now immediate from Corollary \ref{mininjectivesareinj}.

\begin{corollary} \label{commartin}
Let $R$ be a commutative ring. The following statements are equivalent.
\begin{enumerate}
\item Every mininjective $R$-module is injective.
\item $R$ is strongly min-coherent and Artinian.
\end{enumerate}

Moreover, if $Soc(R_{R})$ is almost-injective then the above are equivalent to:
\begin{enumerate}
\item[(3)] $R$ is a quasi-Frobenius serial ring with $\J(R)^{2}=0$.
\end{enumerate}
\end{corollary}
\begin{proof}
$(1)\Leftrightarrow(2)$ Clear by Corollary \ref{mininjectivesareinj}.

$(1)\Rightarrow(3)$ It follows by Theorem \ref{mininjectivesareinjchar} and Proposition \ref{cl}.

$(3)\Rightarrow(1)$ It follows by Theorem \ref{qfring}.
\end{proof}

Following \cite{universallymininjective}, a ring $R$ is called \emph{right universally mininjective} if every right $R$-module is mininjective, equivalently, every minimal right ideal is a direct summand of $R$. However, it turns out that studying the universal mininjective character of a ring doesn't require to check the mininjectivity of every single $R$-module. On the contrary, there exists a class of modules that serve as universal mininjectivity test: the class of simple modules or even the class minimal right ideals. To start with we note that it is clear that $R$ is a right universally mininjective ring if and only if $R$ is a right mininjective and right $PS$-ring. We then have the following.

 \begin{theorem} \label{universiallymininject}
Let $R$ be a ring. The following statements are equivalent.
\begin{enumerate}
\item $R$ is right universally mininjective.
\item Every finitely generated right $R$-module is mininjective.
\item Every cyclic right $R$-module is mininjective.
\item Every simple right $R$-module is mininjective.
\item Every minimal right ideal of $R$ is mininjective.
\item If $aR$ is a minimal right ideal then $Ra$ is a direct summand of $R$.
\item Mininjective right $R$-modules are closed under submodules.
\end{enumerate}
\end{theorem}
\begin{proof} For $(1)\Rightarrow (2)\Rightarrow (3)\Rightarrow(4)\Rightarrow (5)$ and $(1)\Rightarrow (7)$ there is nothing to prove.

$(5)\Rightarrow (1)$ For any minimal right ideal $I$ of $R$ we have $\Ext^{1}_R(R/I,I)=0$, that is, the exact sequence $$0\rightarrow I\rightarrow R\rightarrow R/I\rightarrow 0$$ splits, and so $I$ is a direct summand of $R$.

$(7)\Rightarrow (1)$ Every right $R$-module can be embedded in an injective module so every right $R$-module is mininjective.

$(1)\Rightarrow (6)$ If $aR$ is minimal then $R/Ra$ is min-flat by \cite[Theorem 5.10]{minflat}. Furthermore, $R/Ra$ is indeed projective by \cite[Corollary 3.3]{minflat}, so $Ra$ is a direct summand of $R$.

$(6)\Rightarrow (1)$ If $aR$ is minimal then $R/Ra$ is projective by the hypotheses, so in particular it is min-flat. Then, \cite[Theorem 5.10]{minflat} gives the result.
\end{proof}

In \cite[Theorem 3.8]{almost-v} it was shown that a $2 \times 2$ upper triangular matrix ring over a ring ($UT_{2}(R)$) $R$ is a right almost $V$-ring precisely when $R$ is semisimple. Combining this with the following Theorem we obtain that the ring $R=UT_{n}(k)$ ($k$ is a field) is a right almost $V$-ring if and only if $n=2$. Moreover, the following theorem will also allow us to give an example of a ring which is a right quasi $V$-ring but neither a Kasch nor a right $V$-ring.

\begin{theorem} \label{uppertriangular}
Let $R=UT_{n}(k)$ be the ring of $n \times n$ upper triangular matrices over a field $k$. The following statements are equivalent.
\begin{enumerate}
\item $R$ is a right (resp. left) almost V-ring.
\item Every min-injective right (resp. left) $R$-module is injective.
\item $n=2$.
\end{enumerate}
\end{theorem}
\begin{proof}
$(1)\Rightarrow(2)$ Since $R$ is Artinian, being a right (resp. left) almost V-ring implies that $R$ is serial and $J^{2}(R)=0$ by \cite[Corollary 3.5]{almost-v}. Thus (2) follows by Proposition \ref{qf-serial}.

$(2)\Rightarrow(3)$ We know there is a complete set of simple right $R$-modules $$\left\{\frac{e_{1}R}{e_{1}J(R)},\frac{e_{2}R}{e_{2}J(R)}, \dots, \frac{e_{n-1}R}{e_{n-1}J(R)}, e_{n}R\right\}$$ where the $e_{i}$'s are pairwise orthogonal local idempotents such that $e_{i}J(R)\cong e_{i+1}R$ as $R$-modules for any $i\in \{ 1,2,...,n-1 \}$ (see \cite[p. 366]{lam1}). Of course, $e_{n}R$ is a projective right $R$-module.

If $n\geq 3$ then $e_{2}R/e_{2}J(R)$ cannot be isomorphic to any right ideal of $R$, for otherwise we would have $e_{2}R/e_{2}J(R)\cong e_{n}R$.

Let us show that $e_{2}R/e_{2}J(R)$ is not injective.

Since $e_{1}J(R)\cong e_{2}R$, we can consider a non-zero homomorphism $$f:e_{1}J(R) \rightarrow e_{2}R \rightarrow e_{2}R/e_{2}J(R)$$ ($e_{2}R \rightarrow e_{2}R/e_{2}J(R)$ is the canonical epimorphism). Now, if the diagram $$\xymatrix{ &   e_{1}J(R) \ar@{^{(}->}[r] \ar[d]_{f} & R \ar@{-->}[dl]^{g}  \\ & {\displaystyle \frac{e_{2}R}{e_{2}J(R)}}}$$ could be commutatively completed by $g$, we would have $$g(e_{1}J(R)) \subseteq g(\J(R)) \subseteq \J(e_{2}R/e_{2}J(R)) = 0,$$ contradicting the fact that $f$ is non-zero.

Therefore, $R$ is not a right quasi $V$-ring and this contradicts the hypothesis (2) (see for example Theorem \ref{mininjectivesareinjchar}).

$(3)\Rightarrow(1)$ If $n=2$ then $R$ is Artinian, serial and $J^{2}(R)=0$, so \cite[Theorem 3.8]{almost-v} completes the proof.
\end{proof}

\section*{Acknowledgments}
The work was carried out when the first author was visiting the Almeria University for his postdoctoral research supported by The Scientific and Technological Research Council of Turkey (TUBITAK) under the 2219 - International Postdoctoral Research Fellowship Program for Turkish Citizens. He would like to thank the university for the kind hospitality.

The authors J. R. García Rozas and Luis Oyonarte were partially supported by a project from the Spanish Ministerio de Ciencia e Innovación through its Agencia Estatal de Investigación, REFERENCIA DEL PROYECTO/AEI/PID2020-113552GB-I00 and by the grant 340206-PROYECTO P-FORT-GRUPOS-2023/103 from the University of Almería.


\begin{thebibliography}{00}

\bibitem{onsimpleinj} Y. Alag\"{o}z, S. Benli-G\"oral, E. B\"uy\"uka\c{s}{\i}k, On simple-injective modules, {\it J. Algebra Appl.} {\bf 22(6)} (2023), 2350138.

\bibitem{strongly-simple} I. Amin, Y. Fathi, M. Yousif, Strongly simple-injective rings and modules, {\it Algebra Colloq.} {\bf15(1)} (2008), 135-144.


\bibitem{soc-injective} I. Amin, M. Yousif, N. Zeyada, Soc-injective rings and modules, {\it Comm. Algebra}, {\bf33} (2005), 4229-4250.

\bibitem{almost-v} M. Arabi-Kakavand, S. Asgari, H. Khabazian, Rings for which every simple module is almost injective, {\it Bull. Iranian Math. Soc.} {\bf42(1)} (2016), 113-127.

\bibitem{cl} M. Arabi-Kakavand, S. Asgari, Y. Tolooei, Noetherian rings with almost injective simple modules, {\it Comm. Algebra}, {\bf45(8)} (2017), 3619–3626.

\bibitem{indigent} P. Aydogdu, S. R. Lopez-Permouth, An alternative perspective on injectivity of modules, {\it J. Algebra}, {\bf338} (2011), 207–219.

\bibitem{baba} Y. Baba, Note on almost M-injectives, {\it Osaka Math. J.} {\bf26(3)} (1989), 687–698.

\bibitem{baccella} G. Baccella, Generalized $V$-rings and von Neumann regular rings, {\it Rend. Sem. Mat. Univ. Padova}, {\bf72} (1984), 117-133.


\bibitem{nonnoether} A. Batal, S. Eyidoğan, H. G\"oral, Irreducibility and Primality in Differentiability Classes, {\it Real Anal. Exchange}, {\bf48(1)} (2023), 119-138.


\bibitem{bjork} J. E. Bj\"{o}rk, Rings satisfying certain chain conditions, {\it J. Reine Angew. Math.} {\bf245} (1970), 63-73.

\bibitem{chetnam} T. J. Cheatham, R. S. David, Flat and Projective Character Modules, {\it Proc. Amer. Math. Soc.} {\bf81(2)} (1981), 175-177.

\bibitem{colby} R. R. Colby, Rings which have flat injective modules, {\it J. Algebra}, {\bf35}, (1975) 239–252.

\bibitem{eilenberg-nakayama} S. Eilenberg, T. Nakayama, On the dimension of modules and algebras II (Frobenius algebras and quasi-Frobenius rings), {\it Nagoya Math. J.} {\bf9} (1955), 1-16.

\bibitem{faith} C. Faith, Rings with ascending condition on annihilators, {\it Nagoya Math. J.} {\bf27} (1966), 179-191.

\bibitem{faith-walker} C. Faith, E. A. Walker, Direct-sum representations of injective modules, {\it J. Algebra}, {\bf5} (1967), 203-221.

\bibitem{gobel} R. G\"{o}bel, J. Trlifaj, {\it Approximations and Endomorphism Algebras of Modules}, Walter de Gruyter Expositions in Mathematics,
{\bf 41}, 2006.

\bibitem{harada} M. Harada, Self mini-injective rings, {\it Osaka Math. J.} {\bf19(2)} (1982), 587-597.

\bibitem{harada-noteonalmostrelative} M. Harada, Note on almost relative projectives and almost relative injectives, {\it Osaka J. Math.} {\bf29} (1992), 435-446.

\bibitem{ikeda} M. Ikeda, A characterization of quasi-Frobenius rings, {\it Osaka Math. J.} {\bf4} (1952), 203-209.

\bibitem{ikeda-nakayama} M. Ikeda, T. Nakayama, On some characteristic properties of quasi-Frobenius and regular rings, {\it Proc. Amer. Math. Soc.} \textbf{5} (1954), 15-19.

\bibitem{lam} T. Y. Lam, {\it Lectures on modules and rings}, Springer-Verlag, New York (1999).

\bibitem{lam1} T. Y. Lam, {\it A first course in noncommutative rings}, Springer-Verlag, New York (2001).

\bibitem{minflat} L. Mao, Min-flat modules and min-coherent rings, {\it Comm. Algebra}, {\bf35(2)} (2007), 635-650.

\bibitem{onminflat} L. Mao, On mininjective and min-flat modules, {\it Publ. Math. Debrecen}, {\bf72(3-4)} (2008), 347-358.

\bibitem{maorings} L. Mao, Rings related to mininjective and min-flat modules, {\it Comm. Algebra}, {\bf37(10)} (2009), 3586-3600.

\bibitem{nakayama} T. Nakayama, On Frobeniusean algebras I, {\it Ann. of Math.} {\bf40} (1939), 611-633.

\bibitem{nakayama1} T. Nakayama, On Frobeniusean algebras II, {\it Ann. of Math.} {\bf42} (1941), 1-21.

\bibitem{universallymininjective}  W. K. Nicholson, M. F. Yousif, Mininjective rings, {\it J. Algebra}, {\bf187} (1997), 548-578.

\bibitem{qfrings} W. K. Nicholson, M. F. Yousif, {\it Quasi-Frobenius Rings}, Cambridge Tracts in Math., 158, Cambridge University Press (2003).

\bibitem{neat} G. Renault, $\acute{E}$tude de certains anneaux a li$\acute{e}$s aux sous-modules compl$\acute{e}$ments dun a-module, {\it C. R. Acad. Sci. Paris}, {\bf259} (1964), 4203-4205.

\bibitem{rotman} J. J. Rotman, {\it An Introduction to Homological Algebra}, Academic Press, New York, 1979.

\bibitem{smith} P. F. Smith, Injective modules and prime ideals, {\it Comm. Algebra}, {\bf9(9)} (1981), 989-999.

\bibitem{ware} R. Ware, Endomorphism rings of projective modules, {\it Trans. Amer. Math. Soc.} {\bf155} (1971), 233–256.

\bibitem{zhu2013} Z. Zhu, $\mathfrak{C}$-coherent rings, $\mathfrak{C}$-semihereditary rings and $\mathfrak{C}$-regular rings, {\it Studia Sci. Math. Hungar.} {\bf50(4)} (2013), 491-508.

\bibitem{strongly-min-coherent} Z. Zhu, Strongly $\mathfrak{C}$-coherent rings, {\it Math. Rep.} {\bf19(4)} (2017), 367-380.
\end{thebibliography}
\end{document}